\documentclass[11 pt]{amsart}
\usepackage{graphicx,amsmath,amssymb,amsthm}
\newtheorem{Theorem}{Theorem}[section]

\newtheorem{Corollary}[Theorem]{Corollary}
\newtheorem{Proposition}[Theorem]{Proposition}
\newtheorem{Lemma}[Theorem]{Lemma}

\theoremstyle{definition}

\newtheorem{Remark}[Theorem]{Remark}

\newcommand{\Mcg}{\mathrm{Mod}}
\newcommand{\PMcg}{\mathrm{PMod}}

\begin{document}

\title[Automorphisms of the mapping class group]
{Automorphisms of the mapping class group of a nonorientable
surface}

\author{Ferihe Atalan }
\author{B{\l}a\.{z}ej Szepietowski}
\address{Department of Mathematics, Atilim University,
06836  Ankara, Turkey} \email{fatalan@atilim.edu.tr}
\address{Institute of Mathematics, Faculty of Mathematics, Physics and Informatics, University of Gda\'nsk, 80-308 Gda\'nsk, Poland} \email{blaszep@mat.ug.edu.pl}
\thanks{The first author supported by TUBITAK-110T665. The second author supported by grants 2012/05/B/ST1/02171 and 2015/17/B/ST1/03235 of  National Science Centre, Poland.}
 \pagenumbering{arabic}

\begin{abstract}
Let $S$ be a nonorientable surface of genus $g\ge 5$ with $n\ge 0$ punctures, and $\Mcg(S)$ its
mapping class group. We define the complexity of $S$ to be the maximum rank of a free abelian
subgroup of $\Mcg(S)$. Suppose that $S_1$ and $S_2$ are two such surfaces of the same
complexity. We prove that every isomorphism $\Mcg(S_1)\to\Mcg(S_2)$ is induced by a
diffeomorphism $S_1\to S_2$. This is an analogue of Ivanov's theorem on automorphisms
of the  mapping class  groups of an orientable surface, and also an extension and
improvement of the first author's previous result.
\end{abstract}

\maketitle
\section{Introduction}

Let $\Sigma_{g,b}^n$ (resp. $N_{g,b}^n$)  denote the orientable
(resp. nonorientable) surface of genus $g$ with $b$ boundary
components and $n$ punctures (or distinguished points). If $b$ or
$n$  equals $0$, then we drop it from the notation. Let ${\rm
Mod}(N_{g,b}^n)$ denote the mapping class group of $N_{g,b}^n$,
which is the group of isotopy classes of all diffeomorphisms of
$N_{g,b}^n$, where diffeomorphisms and isotopies are the identity on
the boundary. The mapping class group ${\rm Mod}(\Sigma_{g,b}^n)$ is
defined analogously, but we consider only orientation preserving
maps. The pure mapping class groups $\PMcg(\Sigma_{g,b}^n)$ and
$\PMcg(N_{g,b}^n)$ are the subgroups of $\Mcg(\Sigma_{g,b}^n)$ and
$\Mcg(N_{g,b}^n)$ respectively, consisting of the isotopy classes of
diffeomorphisms fixing each puncture. We denote by  ${\rm
PMod^{+}}(N_{g,b}^n)$ the subgroup of $\PMcg(N_{g,b}^n)$ consisting
of the isotopy classes of diffeomorphisms  preserving local
orientation at each puncture. Finally, let $\mathcal{T}(N_{g,b}^n)$
denote the twist subgroup of ${\PMcg^{+}}(N_{g,b}^n)$ generated by
Dehn twists about all two-sided curves.

We define the {\it complexity} of $N_g^n$, denoted by $\xi\left(N_g^n\right)$,
as the maximum rank of a free abelian subgroup of $\Mcg(N_g^n)$. By \cite{Kuno}, for $g+n>2$ we have
\[\xi\left(N_g^n\right)=
\begin{cases}
\frac{3}{2}(g-1)+n-2 & \textrm{if\ $g$\ is\ odd}\\
\frac{3}{2}g+n-3 & \textrm{if\ $g$\ is\ even.}\\
\end{cases}
\]

The first author proved in \cite{A} that the outer
automorphism group of $\Mcg(N_g)$ is cyclic for $g\ge 5$. In this paper we
improve this result and also extend it to the case of surfaces with punctures.
\begin{Theorem}\label{MainThm1}
For $i=1,2$ let $S_i=N_{g_i}^{n_i}$ be a nonorientable surface of
genus $g_i\ge 5$ with $n_i\ge 0$ punctures, and assume $\xi(S_1)=\xi(S_2)$. Then
every isomorphism
$\Mcg(S_1)\to\Mcg(S_2)$ is induced by a diffeomorphism $S_1\to S_2$.
\end{Theorem}
In particular, for $S_1=S_2$ we obtain the following.
\begin{Corollary}\label{Cor:MainThm1}
The outer automorphism group ${\rm Out( {\rm Mod}}(N_g^n))$ is
trivial for $g\geq 5$ and $n\geq 0$.
\end{Corollary}
The analogous theorem for the mapping class group of an orientable surface
is due to Ivanov \cite{I1}, who proved that if $\Sigma$ is an orientable surface
of genus $g\ge 3$, then every automorphism of $\Mcg(\Sigma)$ is induced by a
diffeomorphism of $\Sigma$, not necessarily orientation preserving. Later,
Ivanov and McCarthy \cite{IM} proved (among other things) that any injective
endomorphism of $\Mcg(\Sigma)$ must be an isomorphism. Finally, by recent
results of Castel \cite{Cast} and Aramayona-Souto \cite{AS}, any nontrivial
endomorphism of $\Mcg(\Sigma)$ must be an isomorphism.
It seems reasonable to expect that Theorem \ref{MainThm1} is true also
for surfaces of genus less than $5$ and sufficiently big complexity. On
the other hand, Corollary \ref{Cor:MainThm1} does not hold for $(g,n)=(2,0)$ or $(3,1)$, see \cite[Proposition 4.5]{A}.

Similarly as in \cite{I1} and \cite{IM}, the main ingredient of
our proof of Theorem \ref{MainThm1} is an algebraic characterization
of Dehn twists (Theorem \ref{Chr-1}), from which we conclude that any
isomorphism  $\Mcg(S_1)\to\Mcg(S_2)$ maps Dehn twists on Dehn twists. However,
unlike for orientable surfaces, $\Mcg(N_g^n)$ is not generated by Dehn
twists (and neither are $\PMcg(N_g^n)$ and $\PMcg^+(N_g^n)$, see \cite{K}
and \cite{SM2}). In Subsection \ref{ss:gens} we fix a finite generating
set of $\PMcg^+(N_g^n)$ consisting of Dehn twists and one crosscap
transposition. By using this generating set we show that any isomorphism
$\Mcg(S_1)\to\Mcg(S_2)$ restricts to an isomorphism
$\PMcg^+(S_1)\to\PMcg^+(S_2)$ of the form $x\mapsto fxf^{-1}$ for
some diffeomorphism $f\colon S_1\to S_2$. Then we conclude Theorem \ref{MainThm1}
by using the following lemma proved in \cite{I1}.
\begin{Lemma}[Ivanov]\label{Ivanov1}
Let $H$ be a normal subgroup of a group $G$ such that the centralizer of $H$ in $G$ is trivial. If
$\varphi\colon G \to G$ is an automorphism such that $\varphi(x)=x$ for all $x\in H$, then $\varphi=\mathrm{id}_{G}$.
\end{Lemma}
We close this introduction by remarking that Corollary \ref{Cor:MainThm1}
together with the fact that the center of $\Mcg(N_g^n)$ is trivial \cite[Corollary 6.3]{SM1},
imply that $\mathrm{Aut}(\Mcg(N_g^n))$ is isomorphic to $\Mcg(N_g^n)$ for $g\ge 5$.

\section{Preliminaries}\label{Prelim}
Let $G$ be a group, $X\subseteq G$ a subset and $x\in G$ an element of
$G$. Then  $C(G)$, $C_G(X)$ and $C_G(x)$ will denote the center of
$G$, the centralizer of $X$ in $G$ and the centralizer of $x$ in
$G$, respectively.

\medskip

Let $g=2\rho+m$ for $\rho\ge 0$, $m\ge 1$. We can represent  $N_g^n$ as an
orientable surface of genus $\rho$ with $n$ punctures and $m$ crosscaps. In
the figures, a crosscap is drawn as a disc with a cross (e.g. Figure \ref{maximality}). This
means that the interior of the disc should be removed from the surface, and then antipodal
points on the resulting boundary component should be identified.

\subsection{Curves and Dehn twists.}
By a \emph{curve} $a$ on a surface $S$ we understand in this paper
an unoriented simple closed curve. According to whether a regular
neighbourhood of $a$ is an annulus or a
M\"obius strip, we call $a$  two-sided or  one-sided  respectively. If $a$ bounds
a disc with at most one puncture or a M\"obius band, then it is called trivial. Otherwise,
we say that it is nontrivial.
Let $S^a$ denote the surface obtained by cutting $S$ along $a$. If
$S^a$ is connected, then we say that $a$ is nonseparating.
Otherwise, $a$ is called separating.
If $a$ is two-sided, then we denote by $t_a$ a Dehn twist about $a$.
On a  nonorientable surface it is impossible to distinguish between
right- and left-handed twists, so the direction of a twist $t_a$ has to be
specified for each curve $a$. Equivalently we may choose an
orientation of a regular neighbourhood of $a$. Then $t_a$ denotes the
right-handed Dehn twist with respect to the chosen orientation. Unless we
specify which of the two twists we mean, $t_a$ denotes any of the
two possible twists.
It is proved in \cite{SM1} that many well known properties of Dehn twists on
orientable surfaces are also satisfied in the nonorientable case. We will use these properties in this paper.

For two curves $a$ and $b$ we denote by $i(a,b)$ their geometric
intersection number (see \cite{SM1} for definition and properties). We say
that  $a$ and $b$ are \emph{equivalent} if there exists a diffeomorphism $h\colon S\to S$ such that $h(a)=b$.

We say that a collection of  curves $\mathcal{C}=\{a_1,\dots,a_k\}$ is
a \emph{multicurve} if the curves $a_i$ are nontrivial,  pairwise disjoint,
pairwise nonisotopic, and none is isotopic to a boundary component of $S$. We
denote by $S^\mathcal{C}$ the surface obtained by cutting $S$ along all curves of $\mathcal{C}$.

\subsection{Pants and skirts.\label{SubS:P-S}} We will use some properties of
pants and skirts (P-S) decompositions  defined in \cite[Section 5]{SM1}.
We say that a multicurve $\mathcal{C}$  is a P-S decomposition
if each $a\in\mathcal{C}$ is two-sided and each component
of $S^\mathcal{C}$ is diffeomorphic to one of the following surfaces:
\begin{itemize}
\item  disc with $2$ punctures (pair of pants of type 1),
\item annulus with $1$ puncture (pair of pants of type 2),
\item sphere with $3$ holes (pair of pants of type 3),
\item M\"obius strip with $1$ puncture (skirt of type 1),
\item M\"obius strip with $1$ hole (skirt of type 2).
\end{itemize}
A P-S decomposition $\mathcal{C}$ is called \emph{separating}
if each $a\in\mathcal{C}$ is a boundary  of two different connected components of $S^\mathcal{C}$.

\begin{Lemma}\label{L:PSdecomp}
Let $S=N_g^n$ for $g\ge 3$, $s=\xi(S)$ if $g\ne 4$, and $s=2+n$ if $g=4$. Suppose
that $a$ is a two-sided  curve on $S$. There exits a P-S
decomposition $\mathcal{C}=\{a_1,\dots,a_s\}$ of $S$, such that
each $a_i$ is equivalent to $a$, if and only if $S^a$ is connected
and nonorientable. Furthermore, if $g+n>3$ then such P-S decomposition must be separating.
\end{Lemma}
\begin{proof} The ``if'' part is left to the reader (see Figure \ref{maximality}).
Suppose that $a$ is separating. Then all $a_i$ are separating. Furthermore,
either each $a_i$  separates a pair of pants of type 1, or  each $a_i$  separates a
skirt of type 1. It follows that $s\le n$, a contradiction. Now suppose that $S^{a_i}$ is
connected and orientable (this is possible only for even $g$). Then every component of $S^\mathcal{C}$ is
a pair of pants of type either 2 or 3. Note, however, that for $i\ne j$ the curves $a_i$, $a_j$ together
separate $S$ (there can be no curve on $S$ disjoint from $a_i$ and intersecting $a_j$ once; such a curve
would be two-sided and one-sided at the same time). It follows that no component of $S^\mathcal{C}$ is a
pair of pants of type 3, hence all components are pairs of pants of type $2$. We have $s\le n$, a contradiction.

Now suppose that $g+n>3$ and $a_i$ is a boundary of only one connected component $P$
of $S^\mathcal{C}$. Because $-\chi(S)=g+n-2>1$, $S^\mathcal{C}$ has more than one
component. It follows that $P$ must be a pair of pants of type 3 and the third
boundary component of $P$ must separate $S$. This is a contradiction, because
all $a_i$ are nonseparating. Thus $\mathcal{C}$ is a separating P-S decomposition of $S$.
\end{proof}
\begin{figure}[hbt]
 \begin{center}
 \includegraphics[width=9cm]{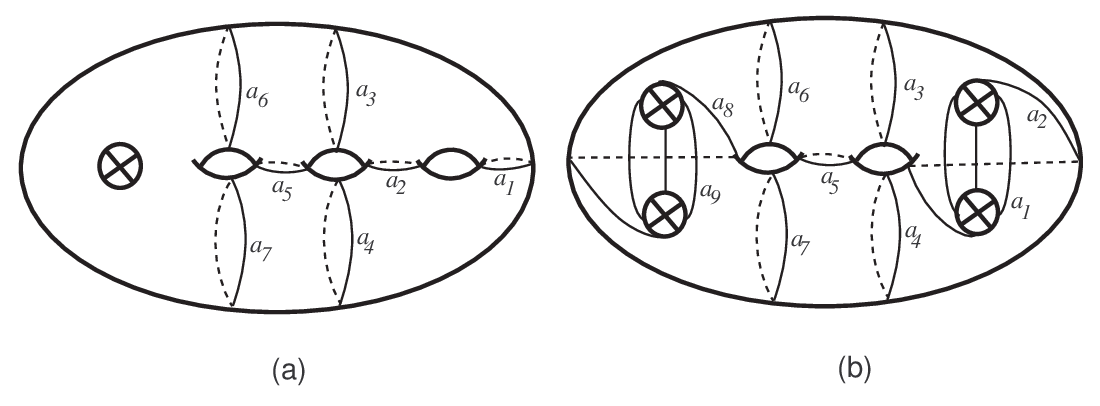}
\caption {P-S decomposition of (a) $N_7$ and (b) $N_8$ satisfying the conditions of  Lemma \ref{L:PSdecomp}.}\label{maximality}
\end{center}
\end{figure}
\begin{Remark}\label{EvenGenusMaxCurv} For $g=4$ there also
exist P-S decompositions of $N_4^n$ of cardinality $\xi(N_4^n)=3+n$. However,
such a P-S decomposition can not consist of nonseparating curves with nonorientable component.
\end{Remark}

\begin{Lemma}\label{L:PSdecomp1}
Let $S=N_g^n$ for $g\ge 3$, $(g,n)\notin\{(3,0),(4,0)\}$ and suppose
that  $\mathcal{C}=\{a_1,\dots,a_s\}$ is a P-S decomposition as in Lemma \ref{L:PSdecomp},
where $s=\xi(S)$ if $g\ne 4$, and $s=2+n$ if $g=4$. For $k\ge 1$ let $T^k_\mathcal{C}$ be
the subgroup of $\Mcg(S)$ generated by $t^k_{a_i}$ for $1\le i\le s$. Then, for each $k\ge 1$:
\begin{itemize}
\item[(a)] $T^k_\mathcal{C}$ is a free abelian group of rank $s$;
\item[(b)] $C_{\Mcg(S)}(T^k_\mathcal{C})=T^1_\mathcal{C}$.
\end{itemize}
\end{Lemma}
\begin{proof}
The  assertion (a) follows from \cite[Proposition 4.4]{SM1}. To prove (b) we use
an idea from the proof of \cite[Theorem 6.2]{SM1}. Suppose $f\in C_{\Mcg(S)}(T^k_\mathcal{C})$. Then
$t^k_{a_i}=ft^k_{a_i}f^{-1}=t^k_{f(a_i)}$ for all $i$. It follows that $f$  fixes each curve $a_i$,
hence it permutes the connected components of $S^\mathcal{C}$. Suppose that $f$ interchanges
some two components $P_1$ and $P_2$ of $S^\mathcal{C}$. By the proof of Lemma \ref{L:PSdecomp},
there are no pairs of pants of type 1 and no skirts of type 1 in the decomposition.  Suppose
that $P_1$ and $P_2$ are skirts of type 2 glued along a curve $a_i$. Then the remaining
boundary curves $a_j\subset P_1$ and  $a_l\subset P_2$ must be glued
together ($a_l=f(a_j)=a_j$), hence $S$ is the closed nonorientable
surface of genus $4$, contrary to the assumptions of the lemma. Similarly, if $P_1$ and $P_2$ are
pairs of pants of type 2 or 3, then $S$ must be a Klein bottle with two
punctures, or a closed nonorientable surface of genus $4$ respectively,
which again  contradicts the assumptions. Thus $f$ fixes each component of
$S^\mathcal{C}$. Furthermore, since $f$ centralizes the boundary twists
of each pair of pants, it preserves its orientation. Because the mapping
class groups of a pair of pants of type 2 or 3, and that of the skirt of
type 2 are generated  by boundary twists, $f$ is a product of some
powers of $t_{a_i}$ for $1\le i\le s$. Thus $C_{\Mcg(S)}(T^k_\mathcal{C})\subseteq T^1_\mathcal{C}$
and the opposite inclusion is obvious.
\end{proof}
Note that (b) of Lemma \ref{L:PSdecomp1} implies that $T^1_\mathcal{C}$ is a maximal abelian subgroup of $\Mcg(S)$.

\subsection{Pure subgroups.} Let $S$ denote the surface  $N_{g}^n$ for $g\ge 3$ and $n\ge 0$. We
recall from \cite{A} the construction of finite index pure subgroups
$\Gamma_m(S)$ of $\Mcg(S)$ (see Section 2 of \cite{A} for more details). Fix an orientable double cover
$\Sigma=\Sigma_{g-1}^{2n}$ of $S$. Then $\Mcg(S)$ can be identified
with the subgroup of $\Mcg(\Sigma)$, consisting of the isotopy
classes of diffeomorphisms commuting with the covering involution.
Consequently, $\Mcg(S)$ acts on $H_1(\Sigma,\mathbb{Z}/m\mathbb{Z})$
for all $m\ge 0$. We define $\Gamma_m(S)$ to be the subgroup of
$\Mcg(S)$ consisting of all elements inducing the identity on
$H_1(\Sigma,\mathbb{Z}/m\mathbb{Z})$. If $m\ge 3$, then $\Gamma_m(S)$
is a pure subgroup of $\Mcg(S)$. In particular, if $f\in\Gamma_m(S)$ preserves a
multicurve $\mathcal{C}$, then $f$ fixes each curve of  $\mathcal{C}$ and, furthermore,
it can be represented by a diffeomorphism equal to the identity on a regular neighbourhood
of each curve of $\mathcal{C}$. If the restriction of $f$ to any connected component
of  $S^{\mathcal{C}}$ is isotopic (by an isotopy that does not have to fix pointwise
the boundary components of $S^{\mathcal{C}}$) either to the identity or to a pseudo-Anosov
map, then $\mathcal{C}$ is called a \emph{reduction system} for $f$. The intersection of
all reduction systems for $f$ is called the \emph{canonical reduction system} for $f$. Reduction
systems were introduced by Birman, Lubotzky and McCarthy in \cite{BLM}, for the case of a nonorientable surface see \cite{Wu}.

\subsection{Algebraic characterization of a Dehn twist.}
The key ingredient of the proof of our main result is an algebraic
characterization of a Dehn twist about a nonseparating curve in the
mapping class group.
Theorem \ref{Chr-1} below is an extension of Theorem 3.1 of \cite{A}
to punctured surfaces. The proof  closely follows
Ivanov's ideas \cite{I1}.

\begin{Theorem}\label{Chr-1}
Let $S=N_g^n$ for $g\geq
3$, $(g,n)\notin\{(3,0),(4,0)\}$ and let $\Gamma$ be
a finite index subgroup of $\Gamma_m(S)$ for $m\ge 3$. An element $f
\in \Mcg(S)$ is a Dehn twist about a nonseparating curve with nonorientable complement if and only
if the following conditions are satisfied:
\begin{itemize}
\item[(i)] $C(C_{\Gamma}(f^{k})) \cong \mathbb{Z}$, for any integer $k \neq 0$
such that $f^{k} \in \Gamma$.
\item[(ii)] Set $s=\xi(S)$ if $g\ne 4$, and $s=2+n$ if $g=4$. There exist elements $f_2,\dots,f_s\in\Mcg(S)$, each conjugate to $f_1=f$, such that
$f_1,\dots,f_s$ generate a free abelian group $K$ of rank $s$.
\item[(iii)] For $k\ge 1$ let $K_k$ be the subgroup of $\Mcg(S)$ generated by $f_i^k$ for
$1\le i\le s$. Then $C_{\Mcg(S)}(K_k)=K$.
\end{itemize}
\end{Theorem}

\begin{proof}
Assume that the above conditions are satisfied, then we have to show
that $f$ is a Dehn twist about a nonseparating curve with nonorientable complement.

Choose any integer $k\neq 0$ such that $f^{k}\in \Gamma$.
Because $f$ has infinite order by (ii), $f^{k}$ is not the identity element.
Let $\mathcal{C}$ be the canonical reduction system for $f^k$. Let
$G$ denote the subgroup generated by the twists about the two-sided
curves in $\mathcal{C}$. Set $G' = G \cap \Gamma$. Then $G$ and
$G'$ are free abelian groups. Firstly, we will show that $G' \subset
C(C_{\Gamma}(f^{k}))$. Let $g \in C_{\Gamma}(f^{k})$. Since $g$
commutes with  $f^{k}$, it preserves the canonical reduction system
$\mathcal{C}$. Because $g$ is pure, it fixes each curve of
$\mathcal{C}$ and also preserves orientation of a regular
neighbourhood of each two-sided curve of $\mathcal{C}$. It follows
that $g$ commutes with each generator $G$, hence $G \subseteq
C_{{\rm M} (S)}(C_{\Gamma}(f^{k}))$. So, $G' \subset
C_{\Gamma}(C_{\Gamma}(f^{k}))= C(C_{\Gamma}(f^{k}))$. For the
last equality observe that, since $f^k\in C_{\Gamma}(f^k)$,
$C_{\Gamma}(C_{\Gamma}(f^{k}))\subseteq C_{\Gamma}(f^{k})$, hence
$C_{\Gamma}(C_{\Gamma}(f^{k}))\subseteq C(C_{\Gamma}(f^{k}))$ and the opposite inclusion is obvious.
The assumption  $C(C_{\Gamma}(f^{k}))=\mathbb{Z}$ implies that
$\mathcal{C}$ contains at most one two-sided curve.

Assume that
$\mathcal{C}$ has no two-sided curve, so that $\mathcal{C} =
\{c_{1}, \cdots, c_{l} \}$, where each $c_{i}$ is a one-sided
curve. Then $S^{\mathcal{C}}$ is connected.
Let $\mathrm{Stab}^+(\mathcal{C})$ be the subgroup of $\Mcg(S)$ consisting of elements
fixing each curve of $\mathcal{C}$ and preserving its orientation. Note that $C_{\Gamma}(f^k)\subseteq\mathrm{Stab}^+(\mathcal{C})$.
The mapping $h\mapsto h_{\mid  S^{\mathcal{C}}}$ defines  an
isomorphism $\mathrm{Stab}^+(\mathcal{C})\to\Mcg(S^\mathcal{C})/{\mathbb{Z}^l}$, where $\mathbb{Z}^l$ is
the subgroup generated by the boundary twists of $S^\mathcal{C}$ (see \cite[Section 4]{Szep_Osaka}). We
also have a monomorphism $\Mcg(S^\mathcal{C})/{\mathbb{Z}^l}\to\Mcg(S')$, where $S'$ is the surface
obtained from $S^\mathcal{C}$ by collapsing each boundary component to a puncture. By composing
these two maps we obtain a monomorphism $\theta\colon\mathrm{Stab}^+(\mathcal{C})\to\Mcg(S')$. Because
$\mathcal{C}$ is the canonical reduction system for $f^k$, $\theta(f^k)$ is either the identity or
pseudo-Anosov. In the former case $f^k$ must be the identity, a contradiction. Suppose $\theta(f^k)$ is
pseudo-Anosov. Set $H=\Gamma\cap K_k$, where $K_k$ is the group from condition (iii). We have
$H\subseteq C_{\Gamma}(f^k)\subseteq\mathrm{Stab}^+(\mathcal{C})$ and $\theta(H)$ is a
free abelian subgroup of $\Mcg(S')$ containing $\theta(f^k)$. Since $\theta(f^k)$ is
pseudo-Anosov, $\theta(H)$ must have rank $1$. This is a contradiction, as $H$ has rank $s>1$.

We have $\mathcal{C}= \{c_{1}, \cdots, c_{l}, a\}$, where $a$ is a
two-sided curve and each $c_i$ is one-sided.
Let  $D$ be the subgroup generated by $f^{k}$ and the twist about
$a$ and denote the intersection $D \cap \Gamma$ by $D'$. Hence, $D'
\subset C(C_{\Gamma}(f^{k}))$ and hence $D'$ is isomorphic to
$\mathbb{Z}$. It follows that $f^{k_1}=t_{a}^{m}$ for some integers
$m$ and $k_1$ (possibly greater than $k$).

Let $f_1,\dots,f_s$ be the elements from condition (ii).
For $1\le i\le s$ we have $f_i^{k_1}=t_{a_i}^m$ for some curve $a_i$ equivalent to
$a_1=a$. We claim that $\mathcal{C}=\{a_1,\dots,a_s\}$ is a P-S decomposition of $S$. If not,
then we can complete $\mathcal{C}$ to a P-S decomposition $\mathcal{C}'$. Let $T_{\mathcal{C'}}$ be
the free abelian group generated by twists about the curves of  $\mathcal{C}'$. We have
$T_{\mathcal{C'}}\subseteq C_{\Mcg(S)}(K_{k_1})=K$. It follows that
$\mathrm{rank}(T_{\mathcal{C'}})\le s$, hence $\mathcal{C'}=\mathcal{C}$.
By (iii) and (b) of Lemma \ref{L:PSdecomp1} we have
$K=C_{\Mcg(S)}(K_{k_1})=C_{\Mcg(S)}(T^m_{\mathcal{C}})=T^1_{\mathcal{C}}$. By
(ii) $f$ is a primitive element of $K=T^1_{\mathcal{C}}$, hence $f=t_{a_1}$. It
follows from Lemma \ref{L:PSdecomp} that $a_1$ is nonseparating and has  nonorientable complement.

The proof of the opposite implication is straightforward and left to
the reader (see \cite{I1}).
\end{proof}

\begin{Corollary}\label{C:TtoT}
For $i=1,2$ let $S_i=N_{g_i}^{n_i}$ for $g_i\ge 5$ and assume $\xi(S_1)=\xi(S_2)$.
Suppose that
$\varphi\colon\Mcg(S_1)\to\Mcg(S_2)$ is an isomorphism. If
$f\in\Mcg(S_1)$ is a Dehn twists about a nonseparating curve with
nonorientable complement, then so is $\varphi(f)$.
\end{Corollary}
\begin{proof} Fix $m\ge 3$.  Because $f$ satisfies the conditions (i), (ii),
(iii) of Theorem \ref{Chr-1} with $\Gamma_m(S_1)$ as $\Gamma$, it follows that
$\varphi(f)$ also satisfies  (i), (ii), (iii) of Theorem \ref{Chr-1} with $\Gamma=\varphi(\Gamma_m(S_1))\cap\Gamma_m(S_2)$.
\end{proof}
\subsection{Chains.}
A sequence $(a_1,\dots,a_k)$ of curves is called a chain if
$i(a_i,a_{i+1})=1$ for $1\le i\le k-1$ and $i(a_i,a_j)=0$ for
$|i-j|>1$.  The integer $k\geq 1$ is called the length of the chain.
If all curves in a chain are two-sided, then a regular
neighbourhood of the union of these curves is orientable. Let
$t_{a_i}$ be right-handed Dehn twists  with respect to some  orientation of
such a neighbourhood for $1\le i\le k$. Then
\begin{itemize}
\item[(a)] $t_{a_i}t_{a_{i+1}}t_{a_i}=t_{a_{i+1}}t_{a_i}t_{a_{i+1}}$ for $1\le i\le k-1$
\item[(b)] $t_{a_i}t_{a_j}=t_{a_j}t_{a_i}$ for $|i-j|>1$.
\end{itemize}
Conversely, if a sequence of Dehn twists $(t_{a_1},\cdots,t_{a_k})$
satisfies (a) and (b), then $(a_1,\dots,a_k)$ is a chain, and the
twists are right-handed with respect to some orientation of a regular
neighbourhood of the union of the curves of the chain (see
\cite[Section 4]{SM1}). A sequence of Dehn twists satisfying (a) and
(b) will also be called a chain. Observe that if $(a_1,a_2)$ is a $2$-chain
of two-sided curves, then $S^{a_i}$ must be  connected and
nonorientable for $i=1,2$.

\subsection{Trees.}
We will now define a tree of curves (and Dehn twists) as a
generalization of a chain. Suppose that $\mathcal{C}$ is a
collection of curves, such that $i(a,b)\in\{0,1\}$ for all $a,b\in\mathcal{C}$. Let $\Gamma(\mathcal{C})$ be a graph with
$\mathcal{C}$ as the set of vertices, and where $a$ and $b$ are
connected by an edge if and only if $i(a,b)=1$. We will call
$\mathcal{C}$ a tree if and only if $\Gamma(\mathcal{C})$ is a tree
(connected and acyclic). If all curves in a tree are two-sided,
then a regular neighbourhood of the union of these curves is
orientable. Let $t_a$ be right-handed Dehn twists  with respect to some
orientation of such a neighbourhood for $a\in\mathcal{C}$. Then
\begin{itemize}
\item[(a')] $t_{a}t_{b}t_{a}=t_{b}t_{a}t_{b}$ if $a$ and $b$ are connected by an edge,
\item[(b')] $t_{a}t_{b}=t_{b}t_{a}$ otherwise.
\end{itemize}
Conversely, suppose that $T=\{t_a\colon a\in\mathcal{C}\}$ is a set
of Dehn twists for some set of curves $\mathcal{C}$, where each two
twists of $T$ either commute, or satisfy the braid relation. Then
the geometric intersection number of the underlying curves is,
respectively, either $0$ or $1$. We say that $T$ is a tree of twists
if and only if $\mathcal{C}$ is a tree.
We will always assume that the curves in $\mathcal{C}$ realize their geometric
intersection number and a  regular neighbourhood of the union of these curves is oriented so that all twists of $T$ are right-handed.

The following corollary follows immediately from Corollary
\ref{C:TtoT}
\begin{Corollary}\label{C:TreetoTree}
For $i=1,2$ let $S_i=N_{g_i}^{n_i}$ for $g_i\ge 5$ and assume $\xi(S_1)=\xi(S_2)$. Suppose that
$\varphi\colon\Mcg(S_1)\to\Mcg(S_2)$ is an isomorphism. If
$T=\{t_a\colon a\in\mathcal{C}\}\subset\Mcg(S_1)$  is a tree of Dehn twists of cardinality at least $2$, then
$\varphi(T)$ is also a tree of Dehn twists for some set of curves
$\mathcal{C}'$, such that $\Gamma(\mathcal{C})$ and
$\Gamma(\mathcal{C}')$ are isomorphic (as abstract graphs).
\end{Corollary}

\subsection{Useful relations among  Dehn twists.}
The following lemma is well-known (see \cite[Proposition
2.12]{LabPar}).
\begin{Lemma}\label{L:A2k+1}
Suppose that $(t_{c_1},t_{c_2},\dots,t_{c_{2k+1}})$ is a chain of twists. Then
\[(t_{c_1}t_{c_2}\cdots t_{c_{2k+1}})^{2k+2}=t_{u_1}t_{u_2},\]
where $t_{u_1}$, $t_{u_2}$ are right-handed twists about the boundary components
of  a regular neighbourhood of the union of the curves $c_i$ (Figure \ref{F:oddchain}).
\end{Lemma}

\begin{figure}[hbt]
  \begin{center}
    \includegraphics[width=8cm]{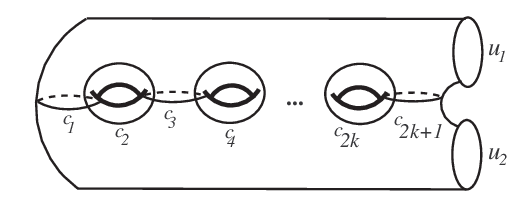}
    \caption{\label{F:oddchain}A chain of two-sided curves of odd length and its regular neighbourhood}
    \label{oddchain}
  \end{center}
\end{figure}

Relations (a) and (b) of the next lemma appear in \cite[Theorem 3.2]{LabPar} as
(R5) and (R6) respectively. Their proof can be deduced from \cite[Proposition 2.12]{LabPar}.

\begin{Lemma}\label{L:D6}
Suppose that  $\{t_{c_0},t_{c_1},\dots,t_{c_7}\}$ is the tree of right-handed
Dehn  twists on $\Sigma_{2,3}$ whose underlying curves are shown on
Figure \ref{3holedsurface}, and $t_{u_i}$,  $i=1,2,3$,  are right-handed
Dehn twists about the boundary components of  $\Sigma_{2,3}$. Then
\begin{align*}
&\mathrm{(a)}\quad t_{u_1}=(t_{c_0}t_{c_1}t_{c_2}t_{c_3}t_{c_4}t_{c_5})^5(t_{c_1}t_{c_2}t_{c_3}t_{c_4}t_{c_5})^{-6}\\
&\mathrm{(b)}\quad t_{u_2}=(t_{c_7}t_{c_6}t_{c_4}t_{c_3}t_{c_2}t_{c_0})^5(t_{c_6}t_{c_4}t_{c_3}t_{c_2}t_{c_0})^{-6}
(t_{c_6}t_{c_5}t_{c_4})^4(t_{c_7}t_{c_6}t_{c_5}t_{c_4})^{-3}
\end{align*}
\end{Lemma}
\begin{figure}[hbt]
  \begin{center}
    \includegraphics[width=7cm]{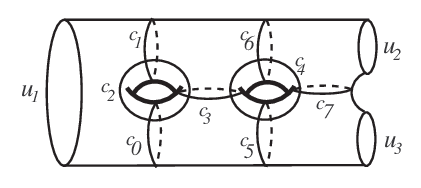}
    \caption{\label{3holedsurface}The curves from Lemma \ref{L:D6}}
  \end{center}
\end{figure}

\subsection{Generators of $\PMcg^+(N_h^n)$.\label{ss:gens}} The aim of this subsection is to fix a
finite generating set of $\PMcg^+(N_h^n)$ for $h\ge 5$. We choose a generating set which differs
slightly from the one given in \cite[Theorem 4.1]{SM2}. We begin its description with Dehn twists.
Let $D$ and $E$ be the trees of curves  from Figures
\ref{tree1odd} and \ref{tree3even}. We will abuse notation and denote by the same
symbols the corresponding  trees of Dehn twists. As we already mentioned in the
introduction, $\PMcg^+(N_h^n)$ is not generated by Dehn twists and to obtain a
generating set for this group we  add to $D$ or $E$ one more generator, namely a
crosscap transposition (in \cite{SM2} a crosscap slide is used). In order to
describe this element, and also to be able to prove Lemmas \ref{L:idonPMcg:odd}
and \ref{L:idonPMcg:even} in Section \ref{OutAut}, we view certain subsurface of $N_h^n$ as a disc with crosscaps.

\begin{figure}[hbt]
  \begin{center}
    \includegraphics[width=8cm]{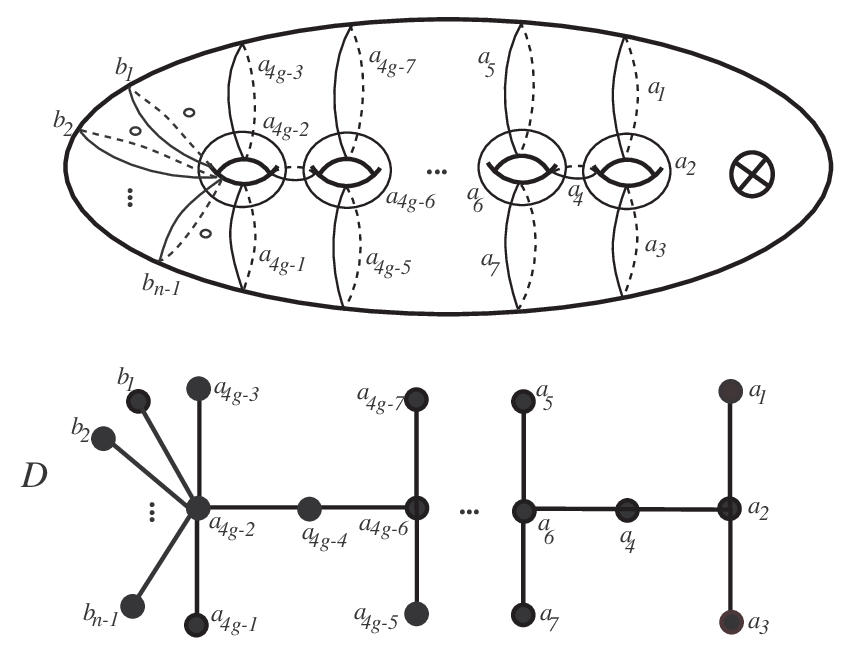}
    \caption{\label{tree1odd}The tree of curves $D$ on $N_{2g+1}^n$.}
  \end{center}
\end{figure}

\begin{figure}[hbt]
  \begin{center}
    \includegraphics[width=8cm]{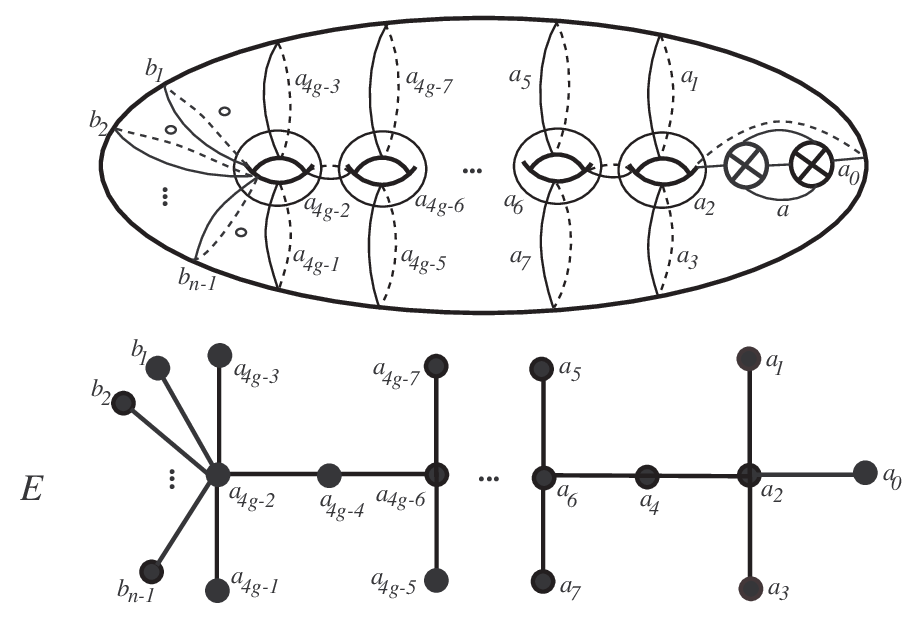}
    \caption{\label{tree3even}The tree of curves $E$ on $N_{2g+2}^n$.}
    \end{center}
\end{figure}

For $k\in\{5,6\}$ let $N_{k,1}$ be a nonorientable surface of genus
$k$ with one boundary component, represented on Figure \ref{gij} as
disc with $k$ crosscaps numbered from $1$ to $k$.
For $i\le j$ let $c_{i,j}$ denote the simple closed curve on
$N_{k,1}$ from Figure \ref{gij}. Note that $c_{i,j}$ is two-sided if and only if $j-i$ is odd. In
such case $t_{c_{i,j}}$ denotes the twist about $c_{i,j}$ in the direction indicated by the  arrows on Figure \ref{gij}.

\begin{figure}[hbt]
 \begin{center}
   \includegraphics{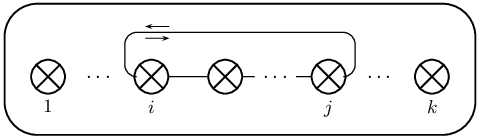}
   \caption{ \label{gij}The surface $N_{k,1}$ and the curve $c_{i,j}$, $k=5$ or $6$.}
 \end{center}
\end{figure}

We denote by $u$ the
\emph{crosscap transposition} defined  to be the isotopy class of
the diffeomorphism of $N_{k,1}$ interchanging the $(k-1)$'st and
$k$'th  crosscaps as shown on Figure \ref{U}, and equal to the
identity outside a disc containing these crosscaps.

\begin{figure}[hbt]
 \begin{center}
   \includegraphics{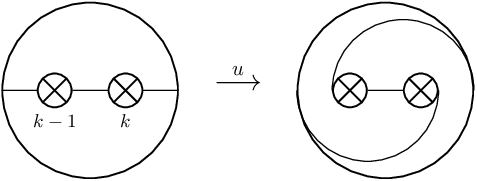}
   \caption{\label{U}The crosscap transposition}
 \end{center}
\end{figure}

\begin{Lemma}\label{L:embNk1} For $g\ge 2$ there are embeddings $\theta_1\colon N_{5,1}\to N_{2g+1}^n$ and
$\theta_2\colon N_{6,1}\to N_{2g+2}^n$, such that:
\begin{itemize}
\item[(a)] for $i=1,2$, $ N_{2g+i}^n\backslash\theta_i(N_{4+i,1})$ is an
orientable surface of genus $g-2$ with $n$ punctures containing the
curves $a_i$ for all $i>8$;
\item[(b)] for $i=1,2$, $a_5=\theta_i(c_{1,2})$, $a_6=\theta_i(c_{2,3})$,
$a_4=\theta_i(c_{3,4})$, $a_2=\theta_i(c_{4,5})$,
$a_1=\theta_i(c_{1,4})$;
\item[(c)] $a_3=\theta_1(t_{c_{4,5}}u^{-1}(c_{1,4}))$;
\item[(d)] $a_0=\theta_2(c_{5,6})$, $a=\theta_2(c_{1,6})$;
\item[(e)] $\theta_2$ maps  boundary curves of a regular neighbourhood
of $c_{1,6}\cup c_{5,6}\cup c_{6,6}$ on $a_1$ and $a_3$.
\end{itemize}
\end{Lemma}
\begin{proof} Suppose $h=2g+1$. Set $c_5=c_{1,2}$, $c_6=c_{2,3}$,
$c_4=c_{3,4}$, $c_2=c_{4,5}$, $c_1=c_{1,4}$ and
$c_3=t_{c_{4,5}}u^{-1}(c_{1,4})$. By changing these curves by a small
isotopy, we may assume that they realize their geometric intersection number.
Then we have $|c_i\cap c_j|=|a_i\cap a_j|$ for all
$i,j\in\{1,\dots,6\}$. Let $M$ (resp. $M'$) be a regular
neighbourhood of the union of $c_i$ (resp. $a_i$) for $i\in\{1,\dots,6\}$. Observe that $M$ and $M'$ are both diffeomorphic to $\Sigma_{2,3}$.
There is a diffeomorphism
$\theta'\colon M\to M'$ such that
$\theta'(c_i)=a_i$ for $i\in\{1,\dots,6\}$.
To see that $\theta'$ can be extended to an
embedding $\theta_1\colon N_{5,1}\to N_{h}^n$ observe that (1) $c_1$, $c_4$
and $c_5$ (resp.    $a_1$, $a_4$ and $a_5$) bound a pair of pants on $N_{5,1}$
(resp. $N_h^n$); (2) $c_3$, $c_4$, $c_5$ and $\partial\!{N_{5,1}}$ bound a $4$-holed
sphere; (3) $c_1$ and $c_3$ (resp. $a_1$ and $a_3$) bound a subsurface of $N_{5,1}$ (resp. $N_h^n$) diffeomorphic to $N_{1,2}$.

Suppose $h=2g+2$. Set $c_5=c_{1,2}$, $c_6=c_{2,3}$, $c_4=c_{3,4}$,
$c_2=c_{4,5}$, $c_1=c_{1,4}$, $c_0=c_{5,6}$. Let $K$ be a regular
neighbourhood of $c_{1,6}\cup c_{5,6}\cup c_{6,6}$. Observe that $K$
is a Klein bottle with two holes, whose one boundary component is
isotopic to $c_1=c_{1,4}$. Let $c_3$ denote the other component of
$\partial\!{K}$. We have $|c_i\cap c_j|=|a_i\cap a_j|$ for all
$i,j\in\{0,\dots,6\}$. Let $M$ (resp. $M'$) be a regular
neighbourhood of the union of $c_i$ (resp. $a_i$) for $i\in\{0,\dots,6\}$. Observe that $M$ and $M'$ are both diffeomorphic to $\Sigma_{2,4}$.
There is a diffeomorphism
$\theta'\colon M\to M'$ such that
$\theta'(c_i)=a_i$ for $i\in\{0,\dots,6\}$. To see that $\theta'$ can be extended to an
embedding $\theta_2\colon N_{6,1}\to N_{h}^n$ observe that (1) $c_1$, $c_4$ and $c_5$
(resp.    $a_1$, $a_4$ and $a_5$) bound a pair of pants on $N_{6,1}$ (resp. $N_h^n$); (2) $c_3$, $c_4$, $c_5$
and $\partial\!{N_{6,1}}$ bound a $4$-holed sphere; (3) two boundary curves of $M$ (resp. $M'$)
bound an annulus with core $c_{1,6}$ (resp. $a$). The conditions  (a, b, d, e) follow immediately from the construction of $\theta_2$.
\end{proof}

Via these embeddings, we will treat $N_{4+i,1}$ as a subsurface of
$N_{2g+i}^n$ for $i=1,2$. Consequently, we will identify curves on
$N_{4+i,1}$ with their images on $N_{2g+i}^n$, and also, using \cite[Corollary 3.8]{Stu_geom}, treat
$\Mcg(N_{4+i,1})$ as a subgroup of $\Mcg(N_{2g+i}^n)$ (in particular $t_{a_5}=t_{c_{1,2}}$ etc.).

\begin{Proposition}\label{P:gensPMCG} For $g\ge 2$,
$\PMcg^+(N_{2g+1}^n)$ (resp. $\PMcg^+(N_{2g+2}^n)$) is generated by
$u$ and $D$ (resp. $u$ and $E$).
\end{Proposition}
\begin{proof}
Let $y=t_{c_{k-1,k}}u$. This element is called crosscap slide and Stukow proved in
\cite[Theorem 4.1]{SM2} that $\PMcg^+(N_{2g+1}^n)$ is generated by $D\cup\{y\}=D\cup\{t_{a_2}u\}$,
whereas $\PMcg^+(N_{2g+2}^n)$ is generated by $E\cup\{y, t_a\}=E\cup\{t_{a_0}u, t_a\}$. It suffices
to show that $t_a$ can be expressed as a product of elements of $E$. This can be achieved by (a) of Lemma \ref{L:D6}:
\[(t_{a_0}t_{a_1}t_{a_2}t_{a_4}t_{a_6}t_{a_5})^5(t_{a_0}t_{a_2}t_{a_4}t_{a_6}t_{a_5})^{-6}=t_a\qedhere\]
\end{proof}
We will also need the following fact about the twist subgroup.
\begin{Lemma}\label{L:gensT}
For $h\ge 5$, $\mathcal{T}(N_h^n)$ is generated by Dehn twists about nonseparating curves with nonorientable complement.
\end{Lemma}
\begin{proof} Set $S=N_h^n$.
By \cite{SM2},  $\mathcal{T}(S)$ is a subgroup of $\PMcg^+(S)$ of index $2$, and
$\PMcg^+(S)=\mathcal{T}(S)\cup u\mathcal{T}(S)$. By Proposition \ref{P:gensPMCG},
$\mathcal{T}(S)$ is generated by $D\cup uDu^{-1}\cup\{u^2\}$ if $h=2g+1$, and by
$E\cup uEu^{-1}\cup\{u^2\}$ if $h=2g+2$. We have $u^2=t_e$, where $e$ is the boundary
curve of the Klein bottle with a hole shown in Figure \ref{U}. Because $D$ and $E$ consist
of Dehn twists about nonseparating curves with nonorientable complement, the same is
true for $uDu^{-1}$ and $uEu^{-1}$, and it suffices to show that $t_e$ can also be
expressed as a product of such twists. Note that $S^e$ is homeomorphic to
$N_{h-2,1}^n\amalg N_{2,1}$. If $h\ge 7$, then the surface $\Sigma_{2,3}$
from Figure \ref{3holedsurface} can be embedded in $S$ so that the boundary
curve $u_1$ of $\Sigma_{2,3}$ coincides with $e$, and then  (a) of Lemma \ref{L:D6}
provides the desired expression of $t_e$ as a product of Dehn twists about nonseparating curves with nonorientable complement.

For the case $h=5,6$ we need the so called star relation, which is a special case
of the fourth relation of \cite[Proposition 2.12]{LabPar}. We say that
curves $c_0, c_1, c_2, c_3$ form a star if $(c_1,c_2,c_3)$ is a chain,
$i(c_0,c_2)=1$ and  $i(c_0,c_1)=i(c_0,c_3)=0$. A regular neighbourhood
of the union of the curves of the star is a 3-holed torus, and we denote
its boundary components by $u_1, u_2, u_3$. The star relation is $(t_{c_0}t_{c_1}t_{c_3}t_{c_2})^3=t_{u_1}t_{u_2}t_{u_3}$,
where the twists are right-handed with respect to some orientation of the
regular neighbourhood. We choose a chain $(c_1,c_2,c_3)$ of curves such
that one of the boundary components of a regular neighbourhood of $c_1\cup c_2\cup c_3$ is
the curve $e$, and we denote the second component by $u_1$. By Lemma \ref{L:A2k+1} we have
$(t_{c_1}t_{c_2}t_{c_3})^4=t_{u_1}t_e$. Note that the connected component of $S^{u_1}$
containing the chain is homeomorphic to $N_{4,1}$ and so we can complete the chain to a
star $(c_0,c_1,c_2,c_3)$, by adding a curve $c_0$, such that one boundary curve of a
regular neighbourhood of the union of the curves of the star is $u_1$ and the other
two components bound M\"obius bands. Then the star relation takes the form  $(t_{c_0}t_{c_1}t_{c_3}t_{c_2})^3=t_{u_1}$ and
$t_e=(t_{c_0}t_{c_1}t_{c_3}t_{c_2})^{-3}(t_{c_1}t_{c_2}t_{c_3})^4$ is the
desired expression of $t_e$ as a product of Dehn twists about nonseparating curves with nonorientable complement.
\end{proof}

\section{Automorphisms of ${\rm {Mod}}(N^n_g)$ }\label{OutAut}
The aim of this section is to prove Theorem \ref{MainThm1}. Our first observation is that we can assume $S_1=S_2$ by the following lemma.
\begin{Lemma}\label{L:diff}
Suppose that $\varphi\colon\Mcg(S_1)\to\Mcg(S_2)$ is an isomorphism, where $S_1$ and
$S_2$ are as in Theorem \ref{MainThm1}. Then $(g_1,n_1)=(g_2,n_2)$.
\end{Lemma}
\begin{proof}
By Lemma \ref{L:gensT} and Corollary \ref{C:TtoT},
$\varphi(\mathcal{T}(S_1))=\mathcal{T}(S_2)$, and hence
$[\Mcg(S_1)\colon\mathcal{T}(S_1)]=[\Mcg(S_2)\colon\mathcal{T}(S_2)]$.
Since $[\Mcg(S_i)\colon\mathcal{T}(S_i)]=2^{n_i+1}n_i!$ by
\cite[Corollary 6.4]{SM2}, we have $n_1=n_2$. This and the equality
$\xi(S_1)=\xi(S_2)$ imply that $g_1$ and $g_2$ must be of the same
parity, and in fact $g_1=g_2$.
\end{proof}
Our next goal is the following  key lemma.
\begin{Lemma}\label{L:almostinner}
Suppose that $h=2g+1$ (resp. $h=2g+2$) for $g\ge 2$ and
$\varphi\colon\Mcg(N_h^n)\to\Mcg(N_h^n)$ is an automorphism. Then
there exists $f\in\Mcg(N_h^n)$ such that $\varphi(t)=ftf^{-1}$ for
each $t\in D$ (resp. $t\in E\cup\{t_a\}$).
\end{Lemma}
After we prove Lemma \ref{L:almostinner}, the next step is to show, using
Proposition \ref{P:gensPMCG}, that the automorphism $\varphi'\colon\Mcg(N_h^n)\to\Mcg(N_h^n)$ defined
as $\varphi'(x)=f^{-1}\varphi(x)f$ restricts to an inner automorphism of $\PMcg^+(N_h^n)$. This step
is completed in Lemmas \ref{L:idonPMcg:odd} and \ref{L:idonPMcg:even}. Finally, we conclude Theorem \ref{MainThm1} by using Lemma \ref{Ivanov1}.

For the proof of Lemma \ref{L:almostinner} we need to compute the centralizers of
sub-trees $\Theta\subset D$ and $\Lambda\subset E$  defined as
\begin{align*}
&\Theta=\{t_{a_1},t_{a_3},t_{a_5}\}\cup\{t_{a_{2i}}\colon 1\le i\le 2g-1\}\cup\{t_{b_j}\colon 1\le j\le n-1\},\\
&\Lambda=\{t_{a_1},t_{a_3},t_{a_5}\}\cup\{t_{a_{2i}}\colon 0\le i\le
2g-1\}\cup\{t_{b_j}\colon 1\le j\le n-1\}.
\end{align*}
Let $\Sigma_{g,n+1}$ (resp.
$\Sigma_{g,n+2}$) be a subsurface of $N_{2g+1}^n$ (resp.
$N_{2g+2}^n$), supporting $D$ (resp. $E$), obtained by removing from
$N_{2g+1}^n$ (resp. $N_{2g+2}^n$) $n$ open discs, each containing
one puncture, and a M\"obius band (resp. an annulus with core $a$).
For $i=1,2$ the inclusion $\Sigma_{g,n+i}\subset N_{2g+i}^n$ induces
a homomorphism $\Mcg(\Sigma_{g,n+i})\to\Mcg( N_{2g+i}^n)$.

\begin{Lemma}\label{L:CTodd}
Suppose that $h=2g+1$ for $g\ge 2$. Then
$C_{\Mcg(N_h^n)}(\Theta)=1$.
\end{Lemma}
\begin{proof}
Let $H$ denote the image of $\Mcg(\Sigma_{g,n+1})$ in $\Mcg(N_h^n)$.
It can be easily deduced from the main result of \cite{LabPar} that
$H$ is generated by twists of $\Theta$.  Thus
$C_{\Mcg(N_h^n)}(\Theta)=C_{\Mcg(N_h^n)}(H)$. Set
$D'=D\backslash\{t_{a_{4i-2}}\colon 1\le i\le g\}$. The curves
supporting the twists of $D'$ form a separating pants and skirts
decomposition of $N_h^n$ (see Subsection \ref{SubS:P-S} for the definition). Let
$h\in C_{\Mcg(N_h^n)}(H)$. Since $D'\subset H$, $h\in
C_{\Mcg(N_h^n)}(D')$. By the proof of (b) of Lemma \ref{L:PSdecomp1},
$h=\prod t_{a_i}^{m_i}$ for some integers $m_i$, where the product
is taken over all $t_{a_i}\in{D'}$. By \cite[Proposition 3.4]{PR},
for every $t_{a_i}\in D'$ there exists a simple closed curve $c$  on
$\Sigma_{g,n+1}$, such that $i(c,a_i)>0$ and  $t_c$ commutes with
all twists in $D'\backslash\{t_{a_i}\}.$ Since $t_c\in H$, it also
commutes with $h$. It follows that $t_c$ commutes with $t_{a_i}^{m_i}$, which is possible only for $m_i=0$, hence $h=1$ and $C_{\Mcg(N_h^n)}(H)$ is
trivial.
\end{proof}

\begin{Lemma}\label{L:CTeven}
Suppose that $h=2g+2$ for $g\ge 2$. Then $C_{\Mcg(N_h^n)}(\Lambda)$
is the infinite cyclic group generated by $t_a$, where $a$ is the curve from Figure \ref{tree3even}.
\end{Lemma}
\begin{proof}
Let $H$ denote the image of $\Mcg(\Sigma_{g,n+2})$ in $\Mcg(N_h^n)$.
Similarly as in the odd genus case, $H$ is generated by twists of
$\Lambda$, thus $C_{\Mcg(N_h^n)}(\Lambda)=C_{\Mcg(N_h^n)}(H)$. Note
that $t_a\in H$, because $a$ is isotopic to a boundary component of
$\Sigma_{g,n+2}$. Set $E'=E\cup\{t_a\}\backslash\{t_{a_{4i-2}}\colon
1\le i\le g\}$. The curves supporting the twists of $E'$ form a
separating P-S decomposition of $N_h^n$.  Let $h\in
C_{\Mcg(N_h^n)}(H)$. By a similar argument as in the proof of (b) of Lemma \ref{L:PSdecomp1},
$h=t_a^m\prod t_{a_i}^{m_i}$ for some integers $m_i$ and $m$, where
the product is taken over all $t_{a_i}\in E'\backslash\{t_a\}$. By
the same argument as in the proof for odd genus, all $m_i=0$, hence
$h=t_a^m$.
\end{proof}

\begin{Lemma}\label{L:noniso_centr}
Let $S=N_{2g+1}^n$ for $g\ge 2$, $n\ge 1$ and suppose
that $\varphi\colon\Mcg(S)\to\Mcg(S)$ is an isomorphism such that $\varphi(t_{a_1})$ and
$\varphi(t_{a_3})$ are Dehn twists about curves $c_1$ and $c_3$. Then $c_1\cup c_3$ does not bound a once-punctured annulus embedded in $S$.
\end{Lemma}
\begin{proof}
Suppose that  $c_1$ and $c_3$ are the boundary curves of a once-punctured annulus embedded in $S$.
Set  $G=C_{\Mcg(S)}\{t_{a_1},t_{a_3}\}$ and $H=\varphi(G)=C_{\Mcg(S)}\{t_{c_1},t_{c_3}\}$. Observe that
$S^{\{c_1,c_3\}}$ is homeomorphic to $N_{2g-1,2}^{n-1}\amalg\Sigma_{0,2}^1$, whereas
$S^{\{a_1,a_3\}}$ is homeomorphic to $\Sigma_{g-1,2}^{n}\amalg N_{1,2}$.
Let $X$ (resp. $Y$) be the subsurface of $S$ homeomorphic to
$\Sigma_{g-1,2}^n$ (resp. $N_{2g-1,2}^{n-1}$) such that
$\partial\!{X}=a_1\cup a_3$ (resp. $\partial\!{Y}=c_1\cup c_3$). The
centralizer $G$ consists of the isotopy classes of diffeomorphisms
of $S$  fixing $a_1$ and $a_3$ whose restriction to $X$ is
orientation preserving. It follows that the inclusion of $X$ in $S$
induces an isomorphism $\Mcg(X)\to G$ (see \cite[\S 5.2]{PSz} or
\cite[\S 4]{Szep_Osaka}). Similarly, the inclusion of $Y$ in $S$
induces an isomorphism $\Mcg(Y)\to H$. Let $K$ denote the image of
$\PMcg(X)$ in $G$. Because $\PMcg(X)$ is generated by Dehn twists
about nonsepataing curves (see \cite[Proposition 2.10]{LabPar}), $K$
is generated by Dehn twists about nonseparating curves with
nonorientable complement. By Corollary \ref{C:TtoT}, $\varphi(K)$ is
also generated by Dehn twists, and hence it is contained in the
image of $\mathcal{T}(Y)$. By \cite[Corollary 6.4]{SM2},
$\mathcal{T}(Y)$ has index $2^n(n-1)!$ in $\Mcg(Y)$, and hence
$[H\colon\varphi(K)]\ge 2^n(n-1)!$. On the other hand
$[H\colon\varphi(K)]=[G\colon K]=[\Mcg(X)\colon\PMcg(X)]=n!$. This
is a contradiction, because $n!<2^n(n-1)!$.
\end{proof}

\begin{proof}[Proof of Lemma \ref{L:almostinner}.] Set $S=N_h^n$.
Suppose $h=2g+1$. By Corollary \ref{C:TreetoTree}, $\varphi(\Theta)$
is a tree of Dehn twists for which the underlying tree of curves is
isomorphic (as abstract graphs) to that of $\Theta$. For
$t_{a_i},t_{b_j}\in\Theta$ choose curves $c_i$, $d_j$ such that
$t_{c_i}=\varphi(t_{a_i})$, $t_{d_j}=\varphi(t_{b_j})$. These
curves  may be chosen to realize their geometric intersection number.

Let $M$ be a closed regular neighbourhood of the union of $c_i$ and
$d_j$ for $t_{c_i}, t_{d_j}\in\varphi(\Theta)$. Note that $M$ is an
orientable surface of genus $g$ with $n+2$ (or $3$ if $n=0$)
boundary components.

\begin{figure}[hbt]
  \begin{center}
    \includegraphics[width=7cm]{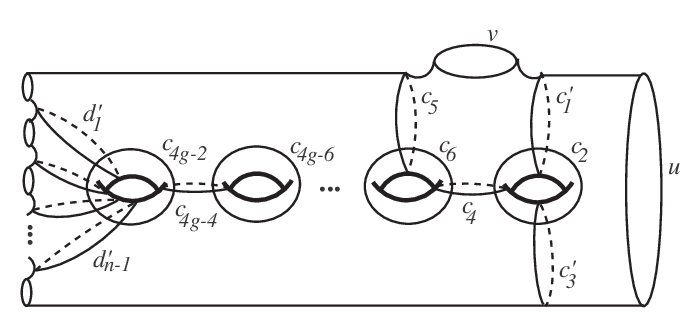}
    \caption{The neighbourhood $M$ supporting $\varphi(\Theta)$.}
    \label{Modd}
  \end{center}
\end{figure}

Similarly, let $M'$ be a closed regular
neighbourhood of the union of the curves supporting $\Theta$. Orient
$M$ and $M'$ so that $t_{a_i}$, $t_{b_j}$  and $t_{c_i}$, $t_{d_j}$
are right-handed Dehn twists. Fix an orientation preserving diffeomorphism
$f_0\colon M'\to M$ such that $f_0(a_{2i})=c_{2i}$ for $1\le 1\le
2g-1$ , $f_0(a_5)=c_5$, $\{f_0(a_1),f_0(a_3)\}=\{c_1,c_3\}$ and
$\{f_0(b_j)\colon 1\le j\le n-1\}=\{d_j\colon 1\le j\le n-1\}$. If
$(g,n)=(2,0)$ then we can also assume $f_0(a_i)=c_i$ for $i=1,3$.
Set $c_i'=f_0(a_i)$ for $i=1,3$ and $d_j'=f_0(b_j)$ for $1\le j\le
n-1$. Either $(c_1',c_3')=(c_1,c_3)$ or  $(c_1',c_3')=(c_3,c_1)$.
Analogously, $(d_1',\dots,d_{n-1}')$ is some (possibly nontrivial)
permutation of $(d_1,\dots,d_{n-1})$. The neighbourhood $M$ and the
curves supporting $\varphi(\Theta)$  are shown on Figure \ref{Modd}.

By  Lemma \ref{L:CTodd},
$C_{\Mcg(S)}(\varphi(\Theta))=\varphi(C_{\Mcg(S)}(\Theta))=1$.
It follows that Dehn twists about the boundary components of $M$ are
trivial, hence each component of $\partial\!{M}$ bounds either a
M\"obius band or a disc with $0$ or $1$ puncture. It is clear that
exactly $1$ component bounds a M\"obius strip, and exactly $n$
components bound once-punctured discs.

Consider the component $u$ of $\partial\!{M}$ which bounds a pair of
pants together with $c'_1$ and $c'_3$.
By Lemma \ref{L:noniso_centr}, $c'_1\cup c'_3$
can not bound a once-punctured annulus in $S$.
It also can not bound a non-punctured annulus, because $t_{c_1}\ne t_{c_3}^{\pm 1}$.
It follows that $u$ bounds a
M\"obius strip.

Suppose $(g,n)\ne (2,0)$ and consider the component $v$ of
$\partial\!{M}$ which bounds a $4$-holed sphere together with $c_5$,
$c_4$ and $c'_1$. For $i=1,3$ set
\[x_i=(t_{a_5}t_{a_6}t_{a_4}t_{a_2}t_{a_i})^6\quad\textrm{and}\quad
y_i=(t_{c_5}t_{c_6}t_{c_4}t_{c_2}t_{c'_i})^6.\] Suppose that
$(c_1',c_3')=(c_3,c_1)$. Then $\varphi(x_3)=y_1$. By Lemma
\ref{L:A2k+1}, $x_3$ is a product of 2 twists commuting with
$t_{a_1}$, whereas $y_1$ does not commute with $t_{c'_3}$, a
contradiction. Hence $c'_i=c_i$ for $i=1,3$. It also follows that
$y_3$ commutes with $t_{c_1}$, which implies that $v$ bounds a
non-punctured disc.

It is now clear that  $f_0$ can be extended to $f\colon S\to
S$. We have $\varphi(t_{a_i})=ft_{a_i}f^{-1}$ for all
$t_{a_i}\in\Theta$. Since each $t_{a_j}\in D$ can be expressed in
terms of $t_{a_i}\in\Theta$, we have
$\varphi(t_{a_j})=ft_{a_j}f^{-1}$ for all $t_{a_j}\in D$. It remains
to prove that $d_i'=d_i$ for $1\le i\le n-1$. We proceed by
induction.

Consider the once-punctured annulus $A_1$, whose boundary is the
union of $b_1$ and $a_{4g-3}$. Let $u_1$ be the boundary of a small
disc contained in $A_1$ and containing the puncture.
By (a) of Lemma \ref{L:D6} we have
\[t_{u_1}=(t_{a_{4g-3}}t_{b_1}t_{a_{4g-2}}t_{a_{4g-4}}t_{a_{4g-6}}t_{a_{4g-7}})^5(t_{b_1}t_{a_{4g-2}}t_{a_{4g-4}}t_{a_{4g-6}}t_{a_{4g-7}})^{-6}.\]
By applying $\varphi$ to the above equality and using Lemma \ref{L:D6} we obtain that
 $\varphi(t_{u_1})$ is
equal to a twist about the curve bounding a disc containing all
punctures of the annulus $A_1'$, whose boundary is the union of
$d_1$ and $f(a_{4g-3})$. Since $t_{u_1}=1$, we have
$\varphi(t_{u_1})=1$. It follows that  $A'_1$ contains only $1$
puncture, hence $d_1=d_1'$.

Now suppose that  $d_i'=d_i$ for $1\le i\le k-1$ for some $k<n$.
Consider the once-punctured annulus $A_k$, whose boundary is the
union of $b_{k-1}$ and $b_k$. Let $u_k$ be the boundary of a small
disc contained in $A_k$ and containing the puncture.
By (b) of Lemma \ref{L:D6} we can express $t_{u_k}$ in terms of Dehn twits
of the tree \[\{t_{b_k}, t_{b_{k-1}}, t_{a_{4g-3}}, t_{a_{4g-2}}, t_{a_{4g-4}}, t_{a_{4g-6}},  t_{a_{4g-7}}\}.\]
By applying $\varphi$ to that expression and using Lemma \ref{L:D6} we
obtain that $\varphi(t_{u_k})$ is equal to a twist about the curve bounding a
disc containing all punctures of the annulus $A_k'$, whose boundary
is the union of $d_k$ and $d_{k-1}$. As above, it follows that
$d_k=d_k'$.

\medskip

For $h=2g+2$ we proceed as above, to obtain a diffeomorphism
$f_0\colon M'\to M$, where $M$ (resp. $M'$) is a regular
neighbourhood of the union of the curves supporting
$\varphi(\Lambda)$ (resp. $\Lambda$), such that $f_0(a_{2i})=c_{2i}$
for $1\le 1\le 2g-1$, $f_0(a_5)=c_5$,
$\{f_0(a_0),f_0(a_1),f_0(a_3)\}=\{c_0,c_1,c_3\}$ and
$\{f_0(b_j)\colon 1\le j\le n-1\}=\{d_j\colon 1\le j\le n-1\}$,
where $\varphi(t_{a_i})=t_{c_i}$ and $\varphi(t_{b_j})=t_{d_j}$. Set
$c_i'=f_0(a_i)$ for $i=0,1,3$ and $d_j'=f_0(b_j)$ for $1\le j\le
n-1$. Note that $M$ and $M'$ are orientable of genus $g$ with $n+3$
(or $4$ if $n=0$) boundary components (see Figure \ref{Meven}).

\begin{figure}[hbt]
  \begin{center}
    \includegraphics[width=7cm]{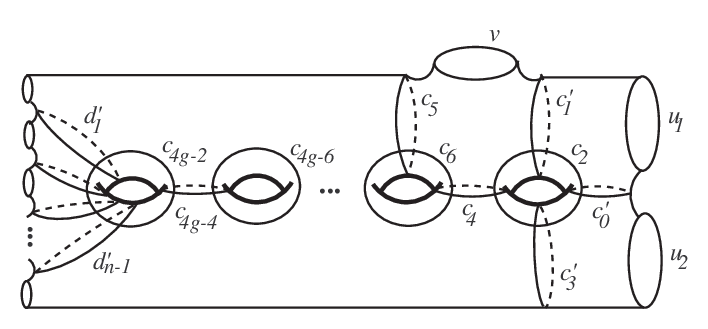}
    \caption{The neighbourhood $M$ supporting $\varphi(\Lambda)$.}
    \label{Meven}
  \end{center}
\end{figure}

For $i\in\{0,1,3\}$ set
\[x_i=(t_{a_5}t_{a_6}t_{a_4}t_{a_2}t_{a_i})^6\quad\textrm{and}\quad
y_i=(t_{c_5}t_{c_6}t_{c_4}t_{c_2}t_{c'_i})^6.\] Suppose
$(g,n)\ne(2,0)$ and consider the component $v$ of $\partial\!{M}$
which bounds a $4$-holed sphere together with $c_5$, $c_4$ and
$c'_1$. It follows from Lemma \ref{L:A2k+1} that
\[\{t_{a_0},t_{a_1},t_{a_3}\}\cap C_{\Mcg(N_h^n)}\{x_0,x_1,x_3\}=\{t_{a_1}\},
\]
hence \[\{t_{c_0},t_{c_1},t_{c_3}\}\cap
C_{\Mcg(N_h^n)}\{y_0,y_1,y_3\}=\{t_{c_1}\}.\] Since neither
$t_{c'_0}$ nor $t_{c'_3}$ commute with $y_1$, we have $c_1=c_1'$.
Furthermore, since $t_{c_1}$ commutes with $y_0$ and $y_3$, $v$
bounds a non-punctured disc.
(If $(g,n)=(2,0)$ then $\{t_{a_0},t_{a_1},t_{a_3}\}\cap
C_{\Mcg(N_h^n)}\{x_0,x_1,x_3\}=\{t_{a_1}, t_{a_3}\}$. It follows
that $c'_0=c_0$, and by composing $f_0$ by a suitable diffeomorphism
if necessary we may assume $c_i'=c_i$ for $i=1,3$. )

By (a) of Lemma \ref{L:D6} we have
\[(t_{a_0}t_{a_1}t_{a_2}t_{a_4}t_{a_6}t_{a_5})^5(t_{a_0}t_{a_2}t_{a_4}t_{a_6}t_{a_5})^{-6}=t_a\]
For $i=0,3$ set
\[z_i=(t_{c'_i}t_{c_1}t_{c_2}t_{c_4}t_{c_6}t_{c_5})^5(t_{c'_i}t_{c_2}t_{c_4}t_{c_6}t_{c_5})^{-6}
\]
Then either $\varphi(t_a)=z_0$ or $\varphi(t_a)=z_3$.
Since $t_a$ commutes with  $t_{a_3}$,
and $z_3$ does not commute with $t_{c'_0}$, we have $\varphi(t_a)=z_0$.
It follows that $c_i'=c_i$ for $i\in\{0,1,3\}$.
Note that $z_0=t_{u_1}$, where $u_1$ is the component of
 $\partial\!{M}$ bounding a pair of pants with $c_0$ and $c_1$.
Let $u_2$ be the the component of
$\partial\!{M}$ bounding a pair of pants with $c_0$ and $c_3$.
By (b) of Lemma \ref{L:D6} we have
\[t_{u_2}=(t_{c_0}t_{c_3}t_{c_2}t_{c_4}t_{c_6}t_{c_5})^5(t_{c_3}t_{c_2}t_{c_4}t_{c_6}t_{c_5})^{-6}
(t_{c_1}t_{c_3}t_{c_2})^4(t_{c_0}t_{c_1}t_{c_3}t_{c_2})^{-3}\]
By applying $\varphi^{-1}$ and using (b) of Lemma \ref{L:D6} again we obtain
\begin{align*}
\varphi^{-1}(t_{u_2})&=(t_{a_0}t_{a_3}t_{a_2}t_{a_4}t_{a_6}t_{a_5})^5(t_{a_3}t_{a_2}t_{a_4}t_{a_6}t_{a_5})^{-6}
(t_{a_1}t_{a_3}t_{a_2})^4(t_{a_0}t_{a_1}t_{a_3}t_{a_2})^{-3}\\
&=t_a^{-1}.
\end{align*}
By Lemma \ref{L:CTeven}
$C_{\Mcg(S)}(\varphi(\Lambda))=\varphi(C_{\Mcg(S)}(\Lambda))$ is the infinite
cyclic group generated by $\varphi(t_a)=t_{u_1}=t_{u_2}^{-1}$.
It follows that $u_1\cup u_2$ bounds an annulus (exterior to $M$) such that the union of $M$
and that annulus is a nonorientable surface of genus $2g+2=h$.

It is clear that $f_0$ can be extended to $f\colon S\to S$.
The rest of the proof follows as in the odd genus case.
\end{proof}

\begin{Lemma}\label{L:centr12}
Let $h=2g+1$ for $g\ge 2$, $D'=D\backslash\{t_{a_i}\colon
i=1,2,3,4\}$ and $H=C_{\Mcg(N_h^n)}(D')$. Let $c$ be the
nontrivial boundary component of a regular neighbourhood of the
union of the curves supporting $D'$. Then $C_H\{t_{a_1},t_{a_2}\}$
is the free abelian group of rank $2$ generated by
$(t_{a_1}t_{a_2})^3$ and either $t_c$ if $(g,n)\ne(2,0)$, or
$(t_{a_5}t_{a_6})^3$ if $(g,n)=(2,0)$.
\end{Lemma}
\begin{proof} Let $d$ be the boundary of a regular neighbourhood of $a_1\cup a_2$ (torus
with one hole) and set $\rho=(t_{a_1}t_{a_2})^3$. Then $\rho^2=t_d$.
Since $t_c$ can be expressed in terms of twists of $D'$, we have
$C_H\{t_{a_1},t_{a_2}\}\subset C_{\Mcg(N_h^n)}\{t_c,t_d\}$. It
follows that any $x\in C_H\{t_{a_1},t_{a_2}\}$  can be represented
by a diffeomorphism, also denoted by $x$, equal to the identity on
regular neighbourhoods of $c$ and $d$. The complement of the union
of such neighbourhoods  has three connected components $S'$, $S''$
and $N$, where $S'$ is diffeomorphic to $\Sigma_{g-1,1}^n$
(containing $a_5$ and $a_6$), $S''$ is diffeomorphic to
$\Sigma_{1,1}$ (containing $a_1$ and $a_2$), and $N$ is diffeomorphic
to $N_{1,2}$. Clearly $x$ preserves each of these components.
Furthermore,  $x$ restricts to a diffeomorphism $x'$ of $S'$, which
commutes with all twists of $D'$ up to isotopy. Since $\PMcg(S')$ is
generated by twists of $D'$,  $x'\in C_{\Mcg(S')}(\PMcg(S'))$. By
\cite[Proposition 5.5 and Theorem 5.6]{PR},
$C_{\Mcg(S')}(\PMcg(S'))=C(\Mcg(S'))$  is the infinite cyclic group
generated either by $t_c$ if $(g,n)\ne (2,0)$, or by
$(t_{a_5}t_{a_6})^3$ if $(g,n)=(2,0)$ (note that
$t_c=(t_{a_5}t_{a_6})^6$ if $(g,n)=(2,0)$). Thus $x'$ is isotopic on
$S'$ to some power of $t_c$ (or $(t_{a_5}t_{a_6})^3$). Analogously,
$x$ restricts to a diffeomorphism $x''$ of $S''$, isotopic on $S''$
to some power of $\rho$. Finally, since $\Mcg(N)$ is generated by the boundary twists,  the restriction of $x$ to $N$ is
isotopic to the product of some power of $t_c$ and some power of
$t_d$.
\end{proof}

\begin{Lemma}\label{L:idonPMcg:odd}
Let $h=2g+1$ for $g\ge 2$ and suppose that $\varphi$ is an automorphism of $\Mcg(N_{h}^n)$ such
that $\varphi(t)=t$ for all $t\in D$. Then $\varphi$ restricts to
the identity on $\PMcg^+(N_{h}^n)$.
\end{Lemma}
\begin{proof}
By Proposition \ref{P:gensPMCG}, it suffices to prove
$\varphi(u)=u$. Let $\mathcal{M}$ be the subgroup of $\Mcg(N_h^n)$
generated by $u$, $t_{a_1}$ and $t_{a_2}$. By \cite[Theorem
4.1]{PSz}, $\mathcal{M}$ is isomorphic to $\Mcg(N_{3,1})$. More
specifically, it is the mapping class group of the nonorientable
subsurface of $N_h^n$ bounded by the curve $c$ from Lemma
\ref{L:centr12}. Set $u_2=u$ and
$u_1=t_{a_2}^{-1}t_{a_1}^{-1}u^{-1}t_{a_1}t_{a_2}$. The following
relations are satisfied in $\mathcal{M}$ (see \cite{PSz}).
\begin{align*}
&(1)\ t_{a_2}t_{a_1}t_{a_2}=t_{a_1}t_{a_2}t_{a_1}\qquad (2)\ u_2u_1u_2=u_1u_2u_1\\
&(3)\  u_2u_1t_{a_2}=t_{a_1}u_2u_1\qquad (4)\ t_{a_2}u_1u_2=u_1u_2t_{a_1}\\
&(5)\ u_it_{a_i}u_i^{-1}=t_{a_i}^{-1}\textrm{\ for\ }i=1,2\qquad
(6)\ u_2t_{a_1}t_{a_2}u_1=t_{a_1}t_{a_2}
\end{align*}
Set $e=t_{a_2}u_2^{-1}t_{a_1}u_2t_{a_2}^{-1}$. Note that  $e$ is a
Dehn twist about the curve $t_{a_2}u_2^{-1}(a_1)=a_3$ (see (b) and (c)
of Lemma \ref{L:embNk1}). In particular $\varphi(e)=e$. Set
$v=eu_1$.
 We have
 \begin{align*}
&e=t_{a_2}u_2^{-1}t_{a_1}u_2t_{a_2}^{-1}\stackrel{(5)}{=}t_{a_2}u_2^{-1}t_{a_1}t_{a_2}u_2
\stackrel{(6)}{=}t_{a_2}t_{a_1}t_{a_2}u_1u_2\\
&v=t_{a_2}t_{a_1}t_{a_2}u_1u_2u_1
\end{align*}
It follows from relations (1,3,4,5) that
$vt_{a_i}v^{-1}=t_{a_i}^{-1}$ for $i=1,2$, and
$v^2=(u_1u_2u_1)^{2}=t_c$ (for the last equality see
\cite[Subsection 3.2]{PSz}). Observe that $v^{-1}\varphi(v)$
commutes with all twists of $D'$, where $D'$ is as in Lemma \ref{L:centr12}, and also with $t_{a_i}$ for
$i=1,2$. Suppose that $(g,n)\ne (2,0)$. By Lemma \ref{L:centr12},
$\varphi(v)=vt_c^k(t_{a_1}t_{a_2})^{3m}$ for some
$k,m\in\mathbb{Z}$. We have $t_c=\varphi(v^2)=t_c^{2k+1}$,
hence $k=0$. If $(g,n)=(2,0)$, then by Lemma \ref{L:centr12},
$\varphi(v)=v(t_{a_5}t_{a_6})^{3k}(t_{a_1}t_{a_2})^{3m}$, and
because $t_c=\varphi(v^2)=t_c^{k+1}$, hence $k=0$. We have
$\varphi(v)=v(t_{a_1}t_{a_2})^{3m}$. It follows that
$\varphi(u_1)=u_1(t_{a_1}t_{a_2})^{3m}$ and
$\varphi(u_2)=(t_{a_1}t_{a_2})^{-3m}u_2$.

Set $t_d=(t_{a_1}t_{a_2})^6$ ($d$ bounds a regular neighbourhood of
$a_1\cup a_2$) and $y=t_{a_2}u_2$. Observe that the curves $y(a_4)$
and  $a_4$ are disjoint up to isotopy (recall $a_4=c_{3,4}$), hence
$yt_{a_4}y^{-1}$ commutes with $t_{a_4}$. By applying $\varphi^2$ we
obtain that $t_d^{-m}y t_{a_4}y^{-1}t_d^{m}$ commutes with
$t_{a_4}$. By \cite[Proposition 4.7]{SM1} it follows that
$i(t_d^{m}(a_4),y(a_4))=0$. We will show that on the other hand $i(t_d^{m}(a_4),y(a_4))\ge 4|m|$, which implies $m=0$ and finishes the proof.
Set $a'_4=y(a_4)$ and  note that $a_4'$, $a_4$ and $a_1$ are pairwise disjoint, and
each of them intersects $a_2$ in a single point. We also have $i(a_4,d)=i(a'_4,d)=2$. Let
$M$ be a regular neighbourhood of $a'_4\cup a_4\cup a_1\cup a_2$, which is  a 3-holed torus (Figure \ref{fig:M}).
The complement of the interior of $M$ in $N_h^n$ is the disjoint union of a
M\"obius band and a subsurface diffeomorphic to $\Sigma_{g-2,2}^n$. In
particular, $M$ is an essential subsurface of $N_h^n$ in the sense of \cite[Definition 3.1]{Stu_geom}, and hence, by
\cite[Proposition 3.3]{Stu_geom}, $i(t_d^{m}(a_4),a'_4)$ is equal
to the geometric intersection number $i_M(t_d^{m}(a_4),a'_4)$ of
$t_d^{m}(a_4)$ and $a'_4$ treated as curves on $M$. Let $\widetilde{M}$
be the 2-holed torus obtained from $M$ by gluing a disc along the
boundary component $f$ (see Figure \ref{fig:M}). Clearly
$i_M(t_d^{m}(a_4),a'_4)\ge i_{\widetilde{M}}(t_d^{m}(a_4),a'_4)$, and
since $a_4'$ is isotopic on $\widetilde{M}$ to $a_4$, we have
$i_{\widetilde{M}}(t_d^{m}(a_4),a'_4)=i_{\widetilde{M}}(t_d^{m}(a_4),a_4)$. Finally, by
\cite[Proposition 3.3]{PR}
$i_{\widetilde{M}}(t_d^{m}(a_4),a_4)=|m|{i_{\widetilde{M}}(d,a_4)}^2=4|m|$. Summarising, we have
\[i(t_d^{m}(a_4),a'_4)=i_M(t_d^{m}(a_4),a'_4)\ge i_{\widetilde{M}}(t_d^{m}(a_4),a'_4)=4|m|
\qedhere
\]
\end{proof}

\begin{figure}[hbt]
 \begin{center}
  \includegraphics[width=5cm]{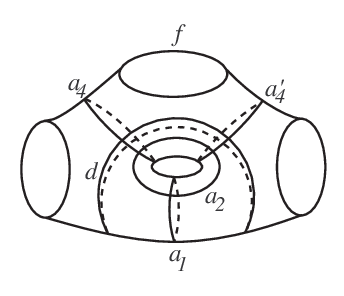}
 \caption{\label{fig:M}The regular neighbourhood $M$ of $a'_4\cup a_4\cup a_1\cup a_2$}
 \end{center}
\end{figure}

\begin{Lemma}\label{L:idonPMcg:even}
Let $h=2g+2$ for $g\ge 2$ and suppose that $\varphi$ is an automorphism of $\Mcg(N_{h}^n)$ such
that $\varphi(t)=t$ for all $t\in E\cup\{t_a\}$. Then $\varphi$
restricts to an inner automorphism of $\PMcg^+(N_{h}^n)$.
\end{Lemma}
\begin{proof}
Let $K$ denote the nonorientable connected component of the surface
obtained by removing from $N_h^n$ an open regular neighbourhood of
$a_1\cup a_3$. Thus, $K$ is a Klein bottle with two holes, and the
other connected component is diffeomorphic to $\Sigma_{g-1,2}^n$.
Furthermore, by (e) of Lemma \ref{L:embNk1}, $K$ is a regular
neighbourhood of $c_{1,6}\cup c_{5,6}\cup c_{6,6}$.  Using \cite[Corollary 3.8]{Stu_geom} we will treat
$\Mcg(K)$  as a subgroup of $\Mcg(N_h^n)$.

Set $u'=\varphi(u)$. Since $u'$ commutes with all twists of $E$
supported on $N_h^n\backslash K$, it can be represented by a
diffeomorphism supported on $K$, by a similar argument as in the
proof of Lemma \ref{L:centr12}. Hence  $u'\in\Mcg(K)$. Since
$u't_{a_0}{u'}^{-1}=t_{a_0}^{-1}$, $u'$ preserves the isotopy class
of $a_0$ by \cite[Proposition 4.6]{SM1}.
Let $\mathcal{S}$ denote the subgroup of $\Mcg(K)$ consisting of
elements fixing the isotopy class of $a_0$, and let $\mathcal{S}^+$
be the subgroup of index $2$ of $\mathcal{S}$ consisting of elements
preserving orientation of a regular neighbourhood of $a_0$. Note
that every element of $\mathcal{S}^+$ can be represented by a
diffeomorphism equal to the identity on a neighbourhood of $a_0$. By
cutting $K$ along $a_0$ we obtain a four-holed sphere, and it
follows from the structure of the mapping class group of this
surface, that $\mathcal{S}^+$ is isomorphic to $\mathbb{Z}^3\times
F_2$, where the factor $\mathbb{Z}^3$ is generated by $t_{a_1}$,
$t_{a_3}$ and $t_{a_0}$, and $F_2$ is the free group of rank $2$
generated by $t_a$ and $ut_au^{-1}$.

Set $v=t_au$. By \cite[Lemma 7.8]{Szep_Osaka} we have
$v^2=t_{a_1}t_{a_3}$. Note that
$v\in\mathcal{S}\backslash\mathcal{S}^+$. It follows from the
previous paragraph, that $\mathcal{S}$ admits a presentation with
generators $t_{a_1}$, $t_{a_0}$, $t_a$ and $v$, and the defining
relations
\begin{align*}
&t_{a_0}t_a=t_at_{a_0},\qquad vt_{a_0}=t_{a_0}^{-1}v,\qquad v^2t_a=t_av^2\\
&t_{a_1}v=vt_{a_1},\qquad t_{a_1}t_{a_0}=t_{a_0}t_{a_1},\qquad
t_{a_1}t_a=t_at_{a_1}
\end{align*}
Let $H$ denote the subgroup generated by $t_{a_0}$, $t_{a_1}$ and
$v^2=t_{a_1}t_{a_3}$. It follows from above presentation that $H$ is
normal in $\mathcal{S}$ and $\mathcal{S}/H$ is isomorphic to the
free product $\mathbb{Z}\ast\mathbb{Z}_2$. More specifically,
denoting by $A$ and $V$ the images in $\mathcal{S}/H$ of
respectively $t_a$ and $v$, we see that $\mathcal{S}/H$  has the
presentation $\langle A, V\,|\,V^2=1\rangle$.

Since $\varphi(t_{a_i})=t_{a_i}$ for $i=0,1,3$ and
$\varphi(t_a)=t_a$  and $u'=\varphi(u)\in\mathcal{S}$, $\varphi$
preserves $\mathcal{S}$ and, by the same argument, $\varphi^{-1}$
also preserves  $\mathcal{S}$, hence $\varphi|_\mathcal{S}$ is an
automorphism of $\mathcal{S}$. Since $\varphi$ is equal to the
identity on $H$, it induces $\phi\in\mathrm{Aut}(\mathcal{S}/H)$. We
have $\phi(A)=A$. Note that every element of order $2$ in
$\mathcal{S}/H$ is conjugate to $V$. In particular $\phi(V)$ is
conjugate to $V$. It is an easy exercise to check, using the normal
form of elements of the free product, that in order for $\phi$ to be
surjective, we must have $\phi(V)=A^nVA^{-n}$ for some
$n\in\mathbb{Z}$.

It follows that
$\varphi(v)=t_{a_1}^kt_{a_3}^lt_{a_0}^{m}t_{a}^nvt_a^{-n}$ for some
integers $l$, $k$ and $m$. We have
$t_{a_1}t_{a_3}=\varphi(v^2)=t_{a_1}^{2k+1}t_{a_3}^{2l+1}$, hence
$k=l=0$. By composing $\varphi$ with the inner automorphism
$x\mapsto t_a^{-n}xt_a^n$ we may assume $n=0$ (note that $t_a$
commutes with all $t_{a_i}$). Thus $\varphi(u)=t_{a_0}^mu$.

Set $y=t_{a_0}u$ and note that $y(a_2)$ is disjoint from $a_2$, hence $yt_{a_2}y^{-1}$ commutes
with $t_{a_2}$. By applying $\varphi$ we obtain that $t_{a_0}^myt_{a_2}y^{-1}t_{a_0}^{-m}$ commutes
with $t_{a_2}$, which gives $i(t_{a_0}^{-m}(a_2),y(a_2))=0$. On the other hand, by a similar
argument as in the proof of Lemma \ref{L:idonPMcg:odd}, we have
$i(t_{a_0}^{-m}(a_2),y(a_2))\ge |m|$,  hence $m=0$.
\end{proof}

\medskip

\begin{proof}[\it Proof of Theorem \ref{MainThm1}]
By Lemma \ref{L:diff} we can assume $S_1=S_2$.
Suppose that $\varphi$ is any automorphism of $\Mcg(N_h^n)$  for $h\ge 5$. By
Lemma \ref{L:almostinner}, there exists $f\in\Mcg(N_h^n)$ such that
$\varphi'$ defined by $\varphi'(x)=f^{-1}\varphi(x)f$ for
$x\in\Mcg(N_h^n)$ is the identity on $D$ (if $h$ is odd) or
$E\cup\{t_a\}$ (if $h$ is even). By Lemma \ref{L:idonPMcg:odd} or
Lemma \ref{L:idonPMcg:even}, $\varphi'$ restricts to an inner
automorphism of $\PMcg^+(N_h^n)$. Thus, by composing $\varphi'$ with
an inner automorphism we obtain $\varphi''$, which restricts to the
identity on $\PMcg^+(N_h^n)$. Since
$C_{\Mcg(N_h^n)}(\PMcg^+(N_h^n))$ is contained in $C_{\Mcg(N_h^n)}(\mathcal{T}(N_h^n))$, it is trivial by \cite[Theorem
6.2]{SM1}. Lemma \ref{Ivanov1} implies that $\varphi''$ is trivial,
hence $\varphi$ is inner.
\end{proof}

\medskip

\noindent{\bf Acknowledgements.} The authors are grateful to Luis Paris and the anonymous referee for valuable comments and suggestions.

\bigskip
\providecommand{\bysame}{\leavevmode\hboxto3em{\hrulefill}\thinspace}

\end{document}